\documentclass[11pt]{article}
\usepackage[utf8]{inputenc}
\usepackage[english]{babel}
\usepackage{cancel}
\usepackage{csquotes}
\usepackage{graphicx}
\usepackage{yhmath}
\usepackage{color}
\usepackage[utf8]{inputenc}
\usepackage{appendix}
\usepackage{amsmath}
\usepackage{amsthm}
\usepackage{amsfonts}
\usepackage{amssymb}
\usepackage{longtable}
\usepackage{cite}
\usepackage{url}
\usepackage{graphics}
\usepackage{subcaption}
\usepackage{epsfig}
\usepackage{geometry}
\usepackage{verbatim}
\usepackage{tikz}
\usetikzlibrary{fit,shapes.misc}

\newcommand\marktopleft[1]{%
    \tikz[overlay,remember picture] 
        \node (marker-#1-a) at (0,2ex) {};%
}
\newcommand\markbottomright[1]{%
    \tikz[overlay,remember picture] 
        \node (marker-#1-b) at (0,0) {};%
    \tikz[overlay,remember picture,thick,dashed,inner sep=3pt]
        \node[draw,rectangle,fit=(marker-#1-a.center) (marker-#1-b.center)] {};%
}

\usepackage{listings,xcolor}
\definecolor{dkgreen}{rgb}{0,.6,0}
\definecolor{dkblue}{rgb}{0,0,.6}
\definecolor{dkyellow}{cmyk}{0,0,.8,.3}

\newtheorem{theorem}{Theorem}[section]

\newtheorem{cor}[theorem]{Corollary}
\newtheorem{prop}[theorem]{Proposition}
\newtheorem{definition}[theorem]{Definition}

\newtheorem{ex}[theorem]{Example}

\title{New scattered sequences of order 3}
\author{Daniele Bartoli\thanks{Department of Mathematics and Informatics, University of Perugia, Perugia, Italy.
Email address: daniele.bartoli@unipg.it}\hspace{0.15 cm}  and Alessandro Giannoni\thanks{Department of Mathematics and Informatics, University of Perugia, Perugia, Italy.
Email address: alessandro.giannoni1@studenti.unipg.it}\hspace{0.15 cm}}

\date{}

\begin{document}
\maketitle

\makeatother
\newcommand{\Prf}{\noindent{\bf Proof}.\quad }
\renewcommand{\labelenumi}{(\alph{enumi})}

\def\as{{^{\sigma}}}
\def\e{{\`{e }}}
\def\a{{\`{a}}}
\def\qi{{^{q^I}}}
\def\qj{{^{q^J}}}
\def\qio{{^{q^{I_0}}}}
\def\qjo{{^{q^{J_0}}}}
\def\qk{{^{q^K}}}
\def\qji{{^{q^{J-I}}}}
\def\qij{{^{q^{I-J}}}}
\def\qjia{{^{q^{J-I}+1}}}
\def\qija{{^{q^{I-J}+1}}}
\def\field{{\mathbb F_{q^n}}}
\def\us{U_{\alpha ,\beta ,\gamma}^{I,J,n}}
\def\fax{{x^{q^I} +\alpha y^{q^J}}}
\def\fbx{{x\qj +\beta z\qi}}
\def\fcx{{y\qi +\gamma z\qj}}
\def\faxl{{\lambda\qi x\qi +\alpha\lambda\qj y\qj}}
\def\fbxl{{\lambda\qj x\qj +\beta\lambda\qi z\qi}}
\def\fcxl{{\lambda\qi y\qi +\gamma\lambda\qj z\qj}}
\newcommand{\fal}[2]{{\lambda\qi #1\qi +\alpha\lambda\qj #2\qj}}
\newcommand{\fbl}[2]{{\lambda\qj #1\qj +\beta\lambda\qi #2\qi}}
\newcommand{\fcl}[2]{{\lambda\qi #1\qi +\gamma\lambda\qj #2\qj}}
\newcommand{\fa}[2]{{#1^{q^I} +\alpha #2^{q^J}}}
\newcommand{\fb}[2]{{#1\qj +\beta #2\qi}}
\newcommand{\fc}[2]{{#1\qi +\gamma #2\qj}}

\begin{abstract}
Scattered sequences are a generalization of scattered polynomials. So far, only scattered sequences of order one and two have been constructed. In this paper an infinite family of scattered sequences of order three is obtained. Equivalence issues are also considered.
\end{abstract}

\bigskip


\section{Introduction}

$h$-scattered sequences and exceptional $h$-scattered sequences can be seen as the geometrical counterparts of exceptional MRD codes.
Rank-metric codes were introduced already in the late 70's by Delsarte \cite{delsarte1978bilinear} and then rediscovered by Gabidulin a few years later \cite{gabidulin1985theory}. Due to their applications in network coding \cite{MR2450762} and cryptography \cite{gabidulin1991ideals,MR3678916}, they attracted many researchers in the last decade.
RM codes are sets of matrices over a finite field $\mathbb{F}_q$ endowed with the so-called rank distance: the distance  between two elements is defined as the rank of their difference. Among them, of particular interest is the family of rank-metric codes whose parameters are optimal, i.e.  they have the maximum possible cardinality for the given minimum rank. Such codes are called \emph{maximum rank distance (MRD) codes} and constructing new families is an important and active research task.
From a different perspective, rank-metric codes can also  be seen as sets of (restrictions of) $\mathbb{F}_q$-linear homomorphisms from $(\mathbb{F}_{q^n})^m$ to $\mathbb{F}_{q^n}$ equipped with the rank distance; see \cite[Sections 2.2 and 2.3]{bartoli2022exceptional}.  In the case of univariate linearized polynomials such a connection was already exploited in \cite{sheekey2016new} by Sheekey, where the notion of scattered polynomials was introduced; see also \cite{bartoli2018exceptional}. Let $f \in \mathcal{L}_{n,q}[X]$ be a $q$-linearized polynomial and let $t$ be a nonnegative integer with $t\leq n-1$. Then, $f$ is said to be \emph{scattered of index $t$} if for every $x,y \in \mathbb{F}_{q^n}^*$
\[ \frac{f(x)}{x^{q^t}}=\frac{f(y)}{y^{q^t}}\,\, \Longleftrightarrow\,\, \frac{y}x\in \mathbb{F}_q, \]
or equivalently
\begin{equation*} \dim_{\mathbb{F}_q}(\ker(f(x)-\alpha x^{q^t}))\leq 1, \,\,\,\text{for every}\,\,\, \alpha \in \mathbb{F}_{q^n}. \end{equation*}
In a more geometrical setting, a scattered polynomial is connected with a scattered subspace of the projective line; see \cite{blokhuis2000scattered}. From a coding theory point of view,  $f$ is scattered of index $t$ if and only if $\mathcal{C}_{f,t}=\langle x^{q^t}, f(x) \rangle_{\mathbb{F}_{q^n}}$ is an MRD code with $\dim_{\mathbb{F}_{q^n}}(\mathcal{C}_{f,t})=2$.
The polynomial $f$ is said to be \emph{exceptional scattered} of index $t$ if it is scattered of index $t$ as a polynomial in $\mathcal{L}_{\ell n,q}[X]$, for infinitely many $\ell$; see \cite{bartoli2018exceptional}. The classification of exceptional scattered polynomials is still not complete, although it gained the attention of several researchers \cite{Bartoli:2020aa4,bartoli2018exceptional,MR4190573,MR4163074,bartoli2022towards}.

So far, many families of scattered polynomials have been constructed; see \cite{sheekey2016new,lunardon2018generalized,lunardon2000blocking,zanella2019condition,MR4173668,longobardi2021linear,longobardi2021large,NPZ,zanella2020vertex,csajbok2018new,csajbok2018new2,marino2020mrd,blokhuis2000scattered}. 

Among them, only two families are exceptional:
\begin{itemize}
    \item[(Ps)] $f(x)=x^{q^s}$ of index $0$, with $\gcd(s,n)=1$ (polynomials of so-called pseudoregulus type);
    \item[(LP)] $f(x)=x+\delta x^{q^{2s}}$ of index $s$, with $\gcd(s,n)=1$ and $\mathrm{N}_{q^n/q}(\delta)\ne1$ (so-called LP polynomials).
\end{itemize}

The generalization of the notion of exceptional scattered polynomials -- together with  their connection with $\mathbb{F}_{q^n}$-linear MRD codes of $\mathbb{F}_{q^n}$-dimension $2$ -- yielded the introduction of the concept of 
 $\mathbb{F}_{q^n}$-linear MRD codes of \emph{exceptional type}; see \cite{bartoli2021linear}. An $\mathbb{F}_{q^n}$-linear MRD code $\mathcal{C}\subseteq\mathcal{L}_{n,q}[X]$ is an \emph{exceptional MRD code} if the rank metric code
\[ \mathcal{C}_\ell=\langle\mathcal{C}\rangle_{\mathbb{F}_{q^{\ell n}}}\subseteq \mathcal{L}_{\ell n,q}[X] \]
is an MRD code for infinitely many $\ell$.

Only two families of exceptional $\mathbb{F}_{q^n}$-linear MRD codes are known:
\begin{itemize}
    \item[(G)] $\mathcal{G}_{k,s}=\langle x,x^{q^s},\ldots,x^{q^{s(k-1)}}\rangle_{\mathbb{F}_{q^n}}$, with $\gcd(s,n)=1$; see \cite{delsarte1978bilinear,gabidulin1985theory,kshevetskiy2005new};
    \item[(T)] $\mathcal{H}_{k,s}(\delta)=\langle x^{q^s},\ldots,x^{q^{s(k-1)}},x+\delta x^{q^{sk}}\rangle_{\mathbb{F}_{q^n}}$, with $\gcd(s,n)=1$ and $\mathrm{N}_{q^n/q}(\delta)\neq (-1)^{nk}$; see  \cite{sheekey2016new,lunardon2018generalized}. 
\end{itemize}
The first family is known as \emph{generalized Gabidulin codes} and the second one as \emph{generalized twisted Gabidulin codes}, whereas in \cite{bartoli2018exceptional} it has been shown that the only exceptional $\mathbb{F}_{q^n}$-linear MRD codes spanned by monomials are the codes (G), in connection with so-called \emph{Moore exponent sets}. Non-existence results on exceptional MRD codes were provided in \cite[Main Theorem]{bartoli2021linear}.

A generalization of MRD codes of exceptional type is connected with the  notions of $h$-scattered sequences and exceptional $h$-scattered sequences, introduced in \cite{bartoli2022exceptional} as sequences of multivariate linearized polynomials $f_1,\ldots, f_s\in \mathcal{L}_{n,q}[X_1,\ldots,X_m]$, such that there exists 
$\mathcal I=(i_1,\ldots, i_m) \in \mathbb{N}^m$ so that the space 
$$\{ (x_1^{q^{i_1}}, \ldots,x_m^{q^{i_m}},f_1(x_1,\ldots,x_m), \ldots,f_s(x_1,\ldots,x_m) ) \, : \, x_1,\ldots,x_m \in \mathbb{F}_{q^n}\}$$
is $h$-scattered. 

Using this new terminology, exceptional $\mathbb{F}_{q^n}$-linear MRD codes correspond to exceptional scattered sequences of order 1. In  \cite{bartoli2022exceptional} exceptional scattered sequences of order 2 were investigated for the first time. Clearly, when considering sequences of order larger than one, one must check that these examples are really new, i.e. they cannot be obtained as direct sum of two scattered sequences of smaller order. This led to the notion of indecomposability; see \cite{bartoli2022exceptional}.

In this paper we consider sequences of order three and we exhibit the first example of indecomposable exceptional scattered sequence. The equivalence issue is also considered and it is worth mentioning that our family is quite large since it contains many non-equivalent sequences.

\section{Main Results}
Let $q=p^h$, where $p$ is a prime and $h>0$ an integer, and denote by $\mathbb{F}_q$ the finite field with $q$ elements. We denote by $\overline{\mathbb{F}}_q$ the algebraic closure of $\mathbb{F}_q$. Finally, $\mathbb{P}^r(\mathbb{F}_q)$ and $\mathbb{A}^r(\mathbb{F}_q)$ denote, respectively, the $r$-dimensional projective and affine space over $\mathbb{F}_q$. In what follows we will make use of basic notions of algebraic curves and hypersurfaces (mainly Bézout's Theorem), when dealing with polynomial systems.  For a more comprehensive introduction to algebraic varieties and curves we refer the interested reader to \cite{HKT,Hartshorne,Stichtenoth}.

We start with the definition of scattered sequences.

\begin{definition}\cite[Definition 3.1]{bartoli2022exceptional}\label{Def:ScatteredSequence}
 Consider $\mathcal{F}=(f_1,\ldots,f_s)$, with $f_1,\ldots,f_s\in \mathcal{L}_{n,q}[\underline{X}]$. We define $$U_{\mathcal{F}}:=\{(f_1(x_1,\ldots,x_m), \ldots,f_s(x_1,\ldots,x_m) ) \, : \, x_1,\ldots,x_m \in \mathbb{F}_{q^n}\}.$$
 
 Let $\mathcal{I}:=(i_1,i_2,\dots,i_m)\in (\mathbb{Z}/n\mathbb{Z})^m$, we define the \textbf{$\mathcal I$-space}
$U_{\mathcal I,\mathcal{F}}:=U_{\mathcal{F}'},$
where
$$\mathcal F'=(X_1^{q^{i_1}},\ldots, X_m^{q^{i_m}}, f_1,\ldots, f_s).$$
The $s$-tuple $\mathcal{F}:=(f_1,\ldots,f_s)$ is said to be an \textbf{$(\mathcal I;h)_{q^n}$-scattered sequence of order $m$} if the $\mathcal I$-space
$U_{\mathcal I,\mathcal{F}}$
is maximum $h$-scattered in $V(m+s,q^{n})$. An $(\mathcal I;h)_{q^n}$-scattered sequence  $\mathcal{F}:=(f_1,\ldots,f_s)$ of order $m$  is said to be \textbf{exceptional} if it is $h$-scattered over infinitely many extensions $\mathbb{F}_{q^{n\ell}}$ of $\mathbb{F}_{q^n}$.

\end{definition}

The main issue, when considering scattered sequences of order larger than one is given by indecomposability.
\begin{definition}
 An $nm$-dimensional $\mathbb{F}_q$-subspace $U_{\mathcal{H}}$ of $V(k,q^n)$   is said to be \textbf{decomposable} if it can be written as
 $$U_{\mathcal{H}}=U_{\mathcal F}\oplus U_{\mathcal G}$$
 for some nonempty $\mathcal F, \mathcal G$. 
 When this happens we say that $\mathcal{F}$ and $\mathcal{G}$ are \textbf{factors} of $\mathcal{H}$. Furthermore, 
 $U$ is then said to be \textbf{indecomposable} if it is not decomposable.
\end{definition}

Let $\mathcal I:=(i_1,\ldots,i_m)$, $\mathcal J:=(j_1,\ldots,j_{m^{\prime}})$, let $\mathcal{F}=(f_1,\ldots,f_{s})$ and $\mathcal{G}=(g_1,\ldots,g_{s^{\prime}})$ be $(\mathcal I;h)_{{q^n}}$ and $(\mathcal J;h)_{{q^n}}$-scattered sequences of orders $m$ and $m^{\prime}$, respectively.
The {direct sum} $\mathcal{H}:=\mathcal{F}\oplus \mathcal{G}$ is the $(s+s^{\prime})$-tuple $(f_1,\ldots,f_s,g_1,\ldots,g_{s^{\prime}})$. Since 
$$U_{\mathcal I\oplus\mathcal J,\mathcal{H}}=U_{\mathcal I,\mathcal{F}}\oplus U_{\mathcal J,\mathcal{G}},$$
$\mathcal{H}$ is an $(\mathcal I \oplus \mathcal J;h)_{q^{n}}$-scattered sequence of order $m+m^{\prime}$.

The main achievement of this paper is the first example of indecomposable exceptional scattered sequences of order three. 

\begin{definition}Let $n$ be a positive integer and $q$ a prime power. Consider the finite field ${\mathbb {F}}_{{q^n}}$. For each choice of  $\alpha ,\beta, \gamma\in{\mathbb {F}}_{{q^n}}^*$ and  $I,J\in\mathbb{N}$, $I<J<n$ we define the set:
$$U_{\alpha ,\beta ,\gamma}^{I,J,n}:=\{(x,y,z, \fa{x}{y}, \fb{x}{z}, \fc{y}{z} )|x,y,z \in{\mathbb {F}}_{{q^n}} \rbrace.$$
\end{definition}
From now on, we will denote $J-I$ as $K$.
\begin{theorem}\label{thmorig1}
Assume that $\gcd(I,J,n)=1$ and that $K_{\alpha ,\beta ,\gamma}^{I,J}:=\alpha\gamma^{q^{K}}\beta^{q^{K}+1}$ is not an $A_{q,I,J}$-power in $\mathbb{F}_{q^n}$, where $A_{q,I,J}:=q^{2 K} +q^{ K}+1$. Then the set $\us$ is scattered.
\end{theorem}
\begin{proof} Let $\lambda\in\mathbb{F}_{q^n}\setminus\mathbb{F}_q$ be such that
$$(x,y,z,\fax,\fbx,\fcx)=\lambda(u,v,w,\fa{u}{v},\fa{u}{w},\fa{v}{w}),$$
with $x,y,z,u,v,w\in\mathbb{F}_q$. $\us$ is scattered if and only if the equation holds for $u=v=w=0$.

By way of contradiction, we assume $(u,v,w)\neq (0,0,0)$. We have
$$\begin{cases}
    x=\lambda u\\
    y=\lambda v\\
    z=\lambda w\\
    \fax=\lambda(\fa{u}{v})\\
    \fbx=\lambda(\fb{u}{w})\\
    \fcx=\lambda(\fc{v}{w}).
\end{cases}$$
So we obtain
$$\begin{cases}
    (\lambda\qi -\lambda)u\qi+(\lambda\qj-\lambda)\alpha v\qj=0\\
    (\lambda\qi -\lambda)\beta w\qi+(\lambda\qj-\lambda) u\qj=0\\
    (\lambda\qi -\lambda)v\qi+(\lambda\qj-\lambda)\gamma w\qj=0.
\end{cases}$$
It is easy to see that a nontrivial solution $(u,v,w)$ must satisfy  $u v w\neq 0$.

We note that this is a linear system in the variables $(\lambda\qi -\lambda),(\lambda\qj-\lambda)$, and since $\lambda\not\in\mathbb{F}_q$ then $(\lambda\qi -\lambda)\neq 0$ or $(\lambda\qj -\lambda)\neq 0$. So this is a linear system of 3 equations in 2 variables which has a nonzero solution. This is possible if and only if
\begin{equation}
    \begin{cases}
    \gamma\beta w^{q^I+q^J}-v\qi u\qj=0\\
    \gamma u\qi w\qj-\alpha v^{q^I+q^J}=0.
\end{cases}\label{thm1.1}
\end{equation}
Letting $z:={w\qi}/{u\qi},y:={v\qi}/{u\qi}$ and dividing the equations of the system by $u^{q^I+q^J}$ we get
    $$
    \begin{cases}
        y=\gamma\beta z\qk\\
        \gamma z\qji-\alpha\gamma\qk\beta\qk z^{(q^{K}+1)^2}=0
    \end{cases}$$
    Since $\gamma,z\neq 0$, we obtain
    $$
    1-\alpha\gamma\qji\beta\qk z^{q^{2K}+q^{K}+1}=0,
    $$
    a contradiction to our hypothesis on $K_{\alpha ,\beta ,\gamma}^{I,J}$.
\end{proof}

\vspace{2mm}

We now prove that $\us$ is exceptional scattered. Denote $({q^{nm}-1})/({q^n-1})$ by $C_{n,m}$.
\begin{prop}\label{propTN} 
    Let $n\in\mathbb{N}$, let $A \in \mathbb{N}$ such that $\gcd(q,A)=1$, then there exist infinitely many $m \in \mathbb{N}$ such that $\gcd(A,C_{n,m})=1$.
\end{prop}
    \begin{proof}
    Consider the factorization of $A=p_1\cdot p_2\cdots p_N$. By the way of contradiction suppose that there are not infinitely many $m \in \mathbb{N}$ such that $\gcd(A,C_{n,m})=1$.
    Therefore, there must exist an $\overline{m}$ such that $$\forall j>0 \text{ there exists an } i_j\text{ such that }p_{i_j}|f(j):=\frac{q^{n(\overline{m}+j)}-1}{q^n-1}=1+q^n+\cdots+q^{n(\overline{m}-1+j)}.$$
We select the primes from the factorization of $A$ that divide at least one $f(j)$ and we denote them by $\{p_1, \ldots, p_M\}$.

We note that \begin{itemize}
    \item if $p_i|f(j)$ then $p_i\not |f(j+1)$ since $p_i\nmid q$;
    \item if $p_i|f(j)$ then $p_i \mid f(kj+(k-1)\overline{m})$ $\forall k>0$.
\end{itemize}
 We prove second property by induction:
 \begin{itemize}
     \item[-] $k=2$: Since $p_i \mid q^{n(\overline{m}+j)}(1+q^n+\dots+q^{n(\overline{m}-1+j)})$, we have that $p_i\mid f(2j+\overline{m})=1+q^n+\dots+q^{n({\overline{m}}-1+j)}+q^{n(\overline{m}+j)}+\cdots+q^{n(2\overline{m}-1+2j)}$.
\item[-]$P(k-1)\to P(k)$:
The proof is analogous to the base case noticing that $f(kj+(k-1)\overline{m})=f((k-1)j+(k-2)\overline{m})+q^{n(k-1)(\overline{m}+j)}f(j).$
 \end{itemize}

 Let $j_i>0$ such that $p_i|f(j_i)$ for $i=1,\dots,M$. This implies that $p_i \mid f(kj_i+(k-1)\overline{m})$ $\forall k>0$.
 Note that  $kj_i+(k-1)\overline{m}=k(j_i+\overline{m})-\overline{m}$ and thus   $p_i \mid f(\prod_{i=1}^M(j_i+\overline{m})-\overline{m})$, for each $ i=1,\dots,M$. So $p_i\nmid f(\prod_{i=1}^M(j_i+\overline{m})-\overline{m}+1)$ for each $i=1,\dots,M$, a contradiction.
 \end{proof}

\begin{cor} \label{cor1}
Assume that $\gcd(I,J,n)=1$ and that
$K_{\alpha ,\beta ,\gamma}^{I,J}$
is not an $A_{q,I,J}$-power in $\mathbb{F}_{q^n}$. Then the set $\us$ is exceptional scattered.
\end{cor}
\begin{proof}
    From the previous proposition, there exist infinitely many $m\in\mathbb{N}$ such that 
    $$\gcd(A_{q,I,J},C_{n,m})=1.$$ 
    Let us consider a fixed $m$ satisfying the above property. By Bézout's identity, there exist integers $c_1$ and $c_2$ such that $c_1A_{q,I,J}+c_2C_{n,m}=1$. Suppose by the way of contradiction that there exists  $\xi \in \mathbb{F}_{q^{mn}}\setminus\mathbb{F}_{q^n}$ such that $K_{\alpha ,\beta ,\gamma}^{I,J}=\xi^{A_{q,I,J}}$. So $\xi^{A_{q,I,J}}\in\mathbb{F}_{q^n}$, and so $1=(\xi^{A_{q,I,J}})^{q^n-1}=(\xi^{q^n-1})^{A_{q,I,J}}$.

Raising both sides to the power of $c_1$, we obtain $$1=(\xi^{q^n-1})^{c_1A_{q,I,J}}=(\xi^{q^n-1})^{-c_2C_{n,m}}(\xi^{q^n-1})=\xi^{q^n-1},$$
a contradiction to $\xi \notin \mathbb{F}_{q^n}$. 

Therefore, there are infinitely many extensions of $\mathbb{F}_q$ where $K_{\alpha ,\beta ,\gamma}^{I,J}$ is not an $A_{q,I,J}$-power, and by  Theorem \ref{thmorig1} the claim follows.
\end{proof}

The evasiveness property will be crucial to prove the indecomposability of $\us$.
\begin{theorem}
    If $K_{\alpha ,\beta ,\gamma}^{I,J}$ is not an $A_{q,I,J}$-power in $\mathbb{F}_{q^n}$ then $\us$ is $(2,2J)_q$-evasive.
\end{theorem}

    \begin{proof} Let 
    \begin{eqnarray*}
    h_1&:=&(x,y,z,\fax,\fbx,\fcx),\\
     h_2&:=&(a,b,c,\fa{a}{b},\fb{a}{c},\fc{b}{c})
     \end{eqnarray*}
   be two $\mathbb{F}_q$-independent vectors of $\us$. A vector $$h_3:=(u,v,w,\fa{u}{v},\fb{u}{w},\fc{v}{w})\in \us $$ belongs to  $\langle{h}_1,{h}_2\rangle$ if and only if the following matrix has rank 2
   $$M:=\begin{pmatrix}
   x&y&z&\fax&\fbx&\fcx\\
   a&b&c&\fa{a}{b}&\fb{a}{c}&\fc{b}{c}\\
   u&v&w&\fa{u}{v}&\fb{u}{w}&\fc{v}{w}
   \end{pmatrix}.
   $$
In what follows we determine the number of $(u,v,w)\in\mathbb{F}_{q^n}^3$ such that $\textrm{rank}(M)=2$.    We will divide the proof into several cases.

\textbf{Case 1.} $xb-ya\neq 0$ 
    
    Assuming that the matrix $M$ has rank 2,  all the $3\times3$ minors  are singular, in particular the ones containing the submatrix $\begin{pmatrix} 
    x&y\\
    a&b \end{pmatrix}$.\\
    So we obtain$$
    \begin{cases}
        u(yc-bz)-v(xc-az)+w(xb-ya)=0\vspace{2mm}\\
        u\begin{vmatrix} y&\fax\\b&\fa{a}{b}\end{vmatrix}-v\begin{vmatrix} x&\fax\\a&\fa{a}{b}\end{vmatrix}+(\fa{u}{v})(xb-ya)=0\vspace{2mm}\\
        u\begin{vmatrix} y&\fbx\\b&\fb{a}{c}\end{vmatrix}-v\begin{vmatrix} x&\fbx\\a&\fb{a}{c}\end{vmatrix}+(\fb{u}{w})(xb-ya)=0\vspace{2mm}\\
        u\begin{vmatrix} y&\fcx\\b&\fc{b}{c}\end{vmatrix}-v\begin{vmatrix} x&\fcx\\a&\fc{b}{c}\end{vmatrix}+(\fc{v}{w})(xb-ya)=0.
    \end{cases}
$$ Dividing by $(xb-ya)$
$$\begin{cases}
    w=vB_1-uA_1\\
    uA_2-vB_2+\fa{u}{v}=0\\
    uA_3-vB_3+u\qj+\beta B_1\qi v\qi-\beta A_1\qi u\qi=0\\
    uA_4-vB_4+\beta B_1\qj v\qj-\beta A_1\qj u\qj=0.
    \end{cases}
$$ 
We notice that we have defined three plane curves
\begin{eqnarray*}
\chi_1&:&u\qi+\alpha v\qj+A_2u-B_2v=0, \\
 \chi_2&:&u\qj+\beta B_1\qi v\qi-\beta A_1\qi u\qi+A_3u-B_3v=0,\\
 \chi_3&:&v\qi+\beta B_1\qj v\qj-\beta A_1\qj u\qj
+A_4u-B_4v=0 \end{eqnarray*}

with coordinates $(u,v,w)$. We can estimate the number of solutions of the previous system by estimating the number of intersections of such curves. Our aim is to use Bézout's theorem, so we need to check common components. 
First we show that $\chi_1$ and $\chi_2$ have no common components. Obviously, if they had a common component, they would also have a common point at infinity, but we see that this is not the case, since \mbox{$\overline{\chi_1}\cap r=\{[1:0:0]\}$} and $\overline{\chi_2}\cap r=\{[0:1:0]\}$, where $r$ is the line at infinity. So we can apply Bézout's theorem on $\chi_1$ and $\chi_2$, so the number of solutions of our system is at most $q^{2J}$. 

Note that $yc-bz\neq 0$ or $xc-az\neq 0$ yield the same result due to the symmetry of the problem. So we can consider this case

\textbf{Case 2.} $xb-ya=yc-bz=xc-az=0$

These equations yield the existence of $\lambda\in\mathbb{F}_{q^n}$ such that $(a,b,c)=\lambda(x,y,z)$.
We notice that $\lambda\not\in\mathbb{F}_q$ because otherwise $h_1$ and $h_2$ would not be independent.
Since $h_3$ is a linear combination of $h_2$ and $h_1$, specifically $(u,v,w)$ is a combination of $(x,y,z)$ and $(a,b,c)=(\lambda x,\lambda y,\lambda z)$. So, there exists $\mu\in\mathbb{F}_{q^n}$ such that $(u,v,w)=\mu (x,y,z)$.
Thus, we need to count how many $\mu\in\mathbb{F}_{q^n}$ make
$$M_2:=\begin{pmatrix}
    x&y&z&\fax&\fbx&\fcx\\
   \lambda x&\lambda y&\lambda z&\faxl&\fbxl&\fcxl\\
   \mu x&\mu y&\mu z&\mu\qi x\qi +\alpha\mu\qj y\qj&\mu\qj x\qj +\beta\mu\qi z\qi&\mu\qi y\qi +\gamma\mu\qj z\qj
\end{pmatrix}$$
a rank-2 matrix.

Clearly, one among $x,y,z$ must be non-zero. Without loss of generality, given the symmetry of the problem, we suppose that  $x\neq 0$. 

\textbf{Case 2.1.} $\lambda\qi\neq\lambda\qj$

Assuming that the matrix M has rank 2, then all the $3\times3$ submatrices  are singular, in particular the ones containing the column $(x,\lambda x,\mu x)$.
$$\begin{pmatrix}
   \marktopleft{c2}\hspace{2mm}x\hspace{2mm}&y&z&\marktopleft{c1}\hspace{6mm}\fax\hspace{6mm}&\fbx&\fcx\\
   \lambda x&\lambda y&\lambda z&\faxl&\fbxl&\fcxl\\
   \mu x\markbottomright{c2}&\mu y&\mu z&\mu\qi x\qi +\alpha\mu\qj y\qj&\mu\qj x\qj +\beta\mu\qi z\qi\markbottomright{c1}&\mu\qi y\qi +\gamma\mu\qj z\qj.
    
\end{pmatrix}$$
By imposing the condition that the highlighted minor is singular, we obtain an equation of the form \begin{equation}
    \mu x \begin{vmatrix}
    \fax&\fbx\\
    \faxl&\fbxl
\end{vmatrix}+K_1\mu\qi+K_2\mu\qj=0 \label{eq2.1.1}\end{equation}

The coefficient of $\mu$ is zero if and only 
\begin{equation}
    x^{q^I+q^J}-\alpha\beta y\qj z\qi=0\label{eq2.1.2}.
\end{equation}
Now, making analogous considerations on the submatrix 
$$\begin{pmatrix}
   \marktopleft{c2}\hspace{2mm}x\hspace{2mm}&y&z&\hspace{6mm}\fax\hspace{6mm}&\marktopleft{c1}\hspace{6mm}\fbx\hspace{6mm}&\fcx\\
   \lambda x&\lambda y&\lambda z&\faxl&\fbxl&\fcxl\\
   \mu x\markbottomright{c2}&\mu y&\mu z&\mu\qi x\qi +\alpha\mu\qj y\qj&\mu\qj x\qj +\beta\mu\qi z\qi&\mu\qi y\qi +\gamma\mu\qj z\qj\markbottomright{c1}
\end{pmatrix}$$
we obtain an equation of the form \begin{equation}
    \mu x \begin{vmatrix}
    \fbx&\fcx\\
    \fbxl&\fcxl
\end{vmatrix}+K_3\mu\qi+K_4\mu\qj=0 \label{eq2.1.3}.\end{equation}
In this case, the coefficient of $\mu$ is zero if and only \begin{equation}
   \beta\gamma z^{q^I+q^J}-x\qj y\qi=0.\label{eq2.1.4}
\end{equation}
Finally, considering the submatrix
$$\begin{pmatrix}
  \marktopleft{c2}\hspace{2mm}x\hspace{2mm}&y&z&\marktopleft{c3}\hspace{6mm}\fax\hspace{6mm}&\hspace{6mm}\fbx\hspace{6mm}&\marktopleft{c1}\hspace{6mm}\fcx\hspace{6mm}\\
   \lambda x&\lambda y&\lambda z&\faxl&\fbxl&\fcxl\\
   \mu x\markbottomright{c2}&\mu y&\mu z&\mu\qi x\qi +\alpha\mu\qj y\qj\markbottomright{c3}&\mu\qj x\qj +\beta\mu\qi z\qi&\mu\qi y\qi +\gamma\mu\qj z\qj\markbottomright{c1}
\end{pmatrix},$$
by arguing as above 
 \begin{equation}
    \mu x \begin{vmatrix}
    \fax&\fcx\\
    \faxl&\fcxl
\end{vmatrix}+K_5\mu\qi+K_6\mu\qj=0 \label{eq2.1.5}.\end{equation}
The coefficient of $\mu$ is zero if and only \begin{equation}
   \gamma x\qi z\qj-\alpha y^{q^I+q^J}=0.\label{eq2.1.6}
\end{equation}
We notice that if at least one of the three coefficients of $\mu$ were non-zero, then one of (\ref{eq2.1.1}), (\ref{eq2.1.3}), (\ref{eq2.1.5}) would be a non-zero polynomial in $\mu$, of degree at most $q^{J}$, and hence it will have no more than $q^{J}$ solutions. If all three coefficients were zero, then from (\ref{eq2.1.2}) we obtain that $y\neq0\neq z$ since $x\neq0$. Also,  by (\ref{eq2.1.4}) and (\ref{eq2.1.6}) 
\begin{equation}
    \begin{cases}
     \beta\gamma z^{q^I+q^J}-x\qj y\qi=0\\
      \gamma x\qi z\qj-\alpha y^{q^I+q^J}=0
         \end{cases}.\label{sys2.1.1}
\end{equation}
Arguing as in the proof of Theorem \ref{thmorig1}, we get a contradiction. 

\textbf{Case 2.2.} $\lambda\qi=\lambda\qj$
We need to consider the following matrix
$$M_3:=\begin{pmatrix}
    x&y&z&\fax&\fbx&\fcx\\
   \lambda x&\lambda y&\lambda z&\lambda\qi x\qi +\alpha\lambda\qi y\qj&\lambda\qi x\qj +\beta\lambda\qi z\qi&\lambda\qi y\qi +\gamma\lambda\qi z\qj\\
   \mu x&\mu y&\mu z&\mu\qi x\qi +\alpha\mu\qj y\qj&\mu\qj x\qj +\beta\mu\qi z\qi&\mu\qi y\qi +\gamma\mu\qj z\qj
    
\end{pmatrix}.$$
Note that $M_3$ has rank two only if 
{\footnotesize$$
\begin{cases}
         -x(\lambda\qi-\lambda)(\mu\qi x\qi+\alpha\mu\qj y\qj)(\fbx)+x(\lambda\qi-\lambda)(\mu\qj x\qj +\beta\mu\qi z\qi)(\fax)=0\\
         -x(\lambda\qi-\lambda)(\mu\qj x\qj+\beta\mu\qi z\qi)(\fcx)+x(\lambda\qi-\lambda)(\mu\qi y\qi+\gamma\mu\qj z\qj)(\fbx)=0\\
          -x(\lambda\qi-\lambda)(\mu\qi x\qi+\alpha\mu\qj y\qj)(\fcx)+x(\lambda\qi-\lambda)(\mu\qi y\qi+\gamma\mu\qj z\qj)(\fax)=0.
\end{cases}
$$}
And thus, we obtain
$$\begin{cases}
    (\mu\qj-\mu\qi)(x^{q^I+q^J}-\alpha\beta y\qj z\qi)=0\\
    (\mu\qj-\mu\qi)(\beta\gamma z^{q^I+q^J}-x\qj y\qi)=0\\
    (\mu\qj-\mu\qi)(\gamma x\qi z\qj-\alpha y^{q^I+q^J})=0
\end{cases}$$
At least one among the above three equations cannot vanish, otherwise $x^{q^I+q^J}-\alpha\beta y\qj z\qi=\beta\gamma z^{q^I+q^J}-x\qj y\qi=\gamma x\qi z\qj-\alpha y^{q^I+q^J}=0$ would yield a contradiction, arguing as before. 
Therefore the above system has at most $q^{J} $ solutions, the claim follows.
\end{proof}

\begin{theorem}
     If $K_{\alpha ,\beta ,\gamma}^{I,J}$ is not an $A_{q,I,J}$-power than $\us$ is $(3,n+2J)_q$-evasive.
\end{theorem}
    \begin{proof} Let $$\begin{array}{c}
    h_1:=(x,y,z,\fax,\fbx,\fcx)\\
    h_2:=(a,b,c,\fa{a}{b},\fb{a}{c},\fc{b}{c})\\
    h_3:=(r,s,t,\fa{r}{s},\fb{r}{t},\fc{s}{t})\end{array}$$
    three $\mathbb{F}_q$-independent vectors of $\us$. A vector of $\us$ $$h_4:=(u,v,w,\fa{u}{v},\fb{u}{w},\fc{v}{w})$$ belongs to $\langle{h}_1,{h}_2,h_3\rangle$ if and only if the following matrix has rank $3$
   $$M:=\begin{pmatrix}
   x&y&z&\fax&\fbx&\fcx\\
   a&b&c&\fa{a}{b}&\fb{a}{c}&\fc{b}{c}\\
   r&s&t&\fa{r}{s}&\fb{r}{t}&\fc{s}{t}\\
   u&v&w&\fa{u}{v}&\fb{u}{w}&\fc{v}{w}
   \end{pmatrix}.
   $$
   We want to count the number of $(u,v,w)\in\mathbb{F}_{q^n}^3$ such that rk($M$)$=3$.

   \textbf{Case 1.} $K:=\begin{vmatrix}
   x\hspace{2mm}&y\hspace{2mm}&z\\a\hspace{2mm}&b\hspace{2mm}&c\\r\hspace{2mm}&s\hspace{2mm}&t\end{vmatrix}\neq 0$ 
    
    Assuming that the matrix M has rank 3, then all the $4\times4$ submatrices  are singular, in particular the ones containing the submatrix $\begin{pmatrix}
    x\hspace{2mm}&y\hspace{2mm}&z\\a\hspace{2mm}&b\hspace{2mm}&c\\r\hspace{2mm}&s\hspace{2mm}&t\end{pmatrix}$.\\
    Thus
    {\footnotesize$$
    \begin{cases}
        u\begin{vmatrix} y&z&\fax\\b&c&\fa{a}{b}\\s&t&\fa{r}{s}\end{vmatrix}-v\begin{vmatrix} x&z&\fax\\a&c&\fa{a}{b}\\r&t&\fa{r}{s}\end{vmatrix}+w\begin{vmatrix} x&y&\fax\\a&b&\fa{a}{b}\\r&s&\fa{r}{s}\end{vmatrix} -K(\fa{u}{v})=0\vspace{2mm}\\
        
        u\begin{vmatrix} y&z&\fbx\\b&c&\fb{a}{c}\\s&t&\fb{r}{t}\end{vmatrix}-v\begin{vmatrix} x&z&\fbx\\a&c&\fb{a}{c}\\r&t&\fb{r}{t}\end{vmatrix}+w\begin{vmatrix} x&y&\fbx\\a&b&\fb{a}{c}\\r&s&\fb{r}{t}\end{vmatrix} -K(\fb{u}{w})=0\vspace{2mm}\\
        
        u\begin{vmatrix} y&z&\fcx\\b&c&\fc{b}{c}\\s&t&\fc{s}{t}\end{vmatrix}-v\begin{vmatrix} x&z&\fcx\\a&c&\fc{b}{c}\\r&t&\fc{s}{t}\end{vmatrix}+w\begin{vmatrix} x&y&\fcx\\a&b&\fc{b}{c}\\r&s&\fc{s}{t}\end{vmatrix} -K(\fc{v}{w})=0\vspace{2mm}.
    \end{cases}
$$}
 Dividing by $K$, we notice that we have defined three hypersurfaces in $\mathbb{P}^{3}(\mathbb{F}_{q^n})$
$$\begin{array}{c}
    \chi_1:uA_1-vB_1+wC_1-(\fa{u}{v})=0\\
    \chi_2:uA_2-vB_2+wC_2-(\fb{u}{w})=0\\
    \chi_3:uA_3-vB_3+wC_3-(\fc{v}{w})=0.
    \end{array}
$$ 
Since the intersection between the three hypersurfaces and the plane at infinity consist in three nonconcurrent lines, they do not share a component. Therefore, $q^{3J}<q^{n+2J}$ since $J<n$.

\textbf{Case 2.} 
$\begin{vmatrix}
   x\hspace{2mm}&y\hspace{2mm}&z\\a\hspace{2mm}&b\hspace{2mm}&c\\r\hspace{2mm}&s\hspace{2mm}&t\end{vmatrix}= 0$.\\
This implies, without loss of generality, the existence of $\lambda_1,\lambda_2\in\mathbb{F}_{q^n}$ such that $(r, s, t)=\lambda_1(x, y, z) + \lambda_2(a, b, c)$. Furthermore, since $(u, v, w)$ is a combination of $(x, y, z), (a, b, c), (r, s, t$), the problem reduces to studying for how many $\mu_1,\mu_2\in\mathbb{F}_{q^n}$ the following matrix $\overline{M}$ has rank 3.

$$\begin{pmatrix}
     
     x&a& \lambda_1x+\lambda_2a&  \mu_1x+\mu_2a\\
     y&b&\lambda_1y+\lambda_2b&\mu_1y+\mu_2b\\
     z&c& \lambda_1z+\lambda_2c&\mu_1z+\mu_2c\\[2mm]
     \fax&\fa{a}{b}&\begin{array}{c}\lambda_1\qi x\qi+\lambda_2\qi a\qi\\+\alpha(\lambda_1\qj y\qj+\lambda_2\qj b\qj)\end{array}&\begin{array}{c}\mu_1\qi x\qi+\mu_2\qi a\qi\\+\alpha(\mu_1\qj y\qj+\mu_2\qj b\qj)\end{array}\\[5mm]
     \fbx&\fb{a}{c}&\begin{array}{c}\lambda_1\qj x\qj+\lambda_2\qj a\qj\\+\beta(\lambda_1\qi z\qi+\lambda_2\qi c\qi)\end{array}&\begin{array}{c}\mu_1\qj x\qj+\mu_2\qj a\qj\\+\beta(\mu_1\qi z\qi+\mu_2\qi c\qi)\end{array}\\[5mm]
     \fcx&\fc{b}{c}&\begin{array}{c}\lambda_1\qi y\qi+\lambda_2\qi b\qi\\+\gamma(\lambda_1\qj z\qj+\lambda_2\qj c\qj)\end{array}&\begin{array}{c}\mu_1\qi y\qi+\mu_2\qi b\qi\\+\gamma(\mu_1\qj z\qj+\mu_2\qj c\qj)\end{array}
     
\end{pmatrix}$$

\textbf{Case 2.1.} $rank\begin{pmatrix}
    x&y&z\\
    a&b&c
\end{pmatrix}=2$\\
Without loss of generality, let's assume that $xb - ya \neq 0$. By  Kronecker's theorem, there exists a non-singular $3\times 3$ submatrix of the following matrix that contains $\begin{pmatrix}
    x&a\\y&b
\end{pmatrix}$

    $$\begin{pmatrix}
    x&a&\lambda_1x+\lambda_2a\\
    y&b&\lambda_1y+\lambda_2b\\
    z&c&\lambda_1z+\lambda_2c\\
    \fax&\fa{a}{b}&\lambda_1\qi x\qi+\lambda_2\qi a\qi+\alpha(\lambda_1\qj y\qj+\lambda_2\qj b\qj)\\
    \fbx&\fb{a}{c}&\lambda_1\qj x\qj+\lambda_2\qj a\qj+\beta(\lambda_1\qi z\qi+\lambda_2\qi c\qi) \\
    \fcx&\fc{b}{c}&\lambda_1\qi y\qi+\lambda_2\qi b\qi+\gamma(\lambda_1\qj z\qj+\lambda_2\qj c\qj)
\end{pmatrix}$$

Denote by $M_1$, $M_2$, and $M_3$ the submatrices containing columns $\{1,2,4\}$, $\{1,2,5\}$, and $\{1,2,6\}$, respectively. Therefore, one among $K_1:=\det M_1, K_2:=\det M_2, K_3:=\det M_3$ must be nonzero.
\begin{itemize}
    \item $K_1\neq 0$ or $K_2\neq 0$\\
    Since
    $$\begin{vmatrix}
     x&y&\fax&\fbx\\
   a&b&\fa{a}{b}&\fb{a}{c}\\
   \lambda_1x+\lambda_2a&\lambda_1y+\lambda_2b&\begin{array}{c}\lambda_1\qi x\qi+\lambda_2\qi a\qi\\+\alpha(\lambda_1\qj y\qj+\lambda_2\qj b\qj)
   \end{array}&\begin{array}{c}\lambda_1\qj x\qj+\lambda_2\qj a\qj\\+\beta(\lambda_1\qi z\qi+\lambda_2\qi c\qi)
   \end{array}\\
  \mu_1x+\mu_2a&\mu_1y+\mu_2b&\begin{array}{c}\mu_1\qi x\qi+\mu_2\qi a\qi\\+\alpha(\mu_1\qj y\qj+\mu_2\qj b\qj)
   \end{array}&\begin{array}{c}\mu_1\qj x\qj+\mu_2\qj a\qj\\+\beta(\mu_1\qi z\qi+\mu_2\qi c\qi)
   \end{array}
\end{vmatrix}=0,$$
we obtain 
\begin{equation}\label{Eq:polinomio}
    D_1\mu_1+D_2\mu_2+D_3\mu_1\qi+D_4\mu_2\qi+\mu_1\qj(-\alpha y\qj K_2+x\qj K_1)+\mu_2\qj(-\alpha b\qj K_2+a\qj K_1)=0,
\end{equation}
for some $D_1,D_2,D_3,D_4$.
We note that this polynomial cannot vanish, otherwise 
$$\begin{cases}
    x\qj K_1=\alpha y\qj K_2\\
    a\qj K_1=\alpha b\qj K_2
\end{cases}$$
and by dividing by the nonzero $K_i$ and taking the $q^j$-th root, we obtain that the matrix $\begin{pmatrix}
    x&y\\a&b
\end{pmatrix}$ is singular, which contradicts the assumption.
Therefore, \eqref{Eq:polinomio} is a non-vanishing polynomial in two variables, of degree less than or equal to $q^{J}$, so it has no more than $q^{n+J} $ solutions.
\item $K_1=0=K_2$ and $K_3\neq 0$\\
Since 
$$\begin{vmatrix}
     x&y&\fax&\fcx\\
   a&b&\fa{a}{b}&\fc{b}{c}\\
   \lambda_1x+\lambda_2a&\lambda_1y+\lambda_2b&\begin{array}{c}\lambda_1\qi x\qi+\lambda_2\qi a\qi\\+\alpha(\lambda_1\qj y\qj+\lambda_2\qj b\qj)
   \end{array}&\begin{array}{c}\lambda_1\qi y\qi+\lambda_2\qi b\qi\\+\gamma(\lambda_1\qj z\qj+\lambda_2\qj c\qj)
   \end{array}\\
  \mu_1x+\mu_2a&\mu_1y+\mu_2b&\begin{array}{c}\mu_1\qi x\qi+\mu_2\qi a\qi\\+\alpha(\mu_1\qj y\qj+\mu_2\qj b\qj)
   \end{array}&\begin{array}{c}\mu_1\qi y\qi+\mu_2\qi b\qi\\+\gamma(\mu_1\qj z\qj+\mu_2\qj c\qj)
   \end{array}
\end{vmatrix},$$
we obtain 
\begin{equation}
    E_1\mu_1+E_2\mu_2-x\qi K_3\mu_1\qi-a\qi K_3\mu_2\qi-\alpha K_3 y\qj\mu_1\qj-\alpha K_3 b\qj\mu_2\qj=0,
\end{equation}
for some $E_1,E_2$.
Since $xb - ya \neq 0$, either $x$ or $a$ is nonzero. Therefore, the polynomial is nonzero, and proceeding as before, the claim follows.
\end{itemize}
\textbf{Case 2.2.}  $rank\begin{pmatrix}
    x&y&z\\
    a&b&c
\end{pmatrix}=1$\\
This implies the existence of $\lambda,\sigma,\mu\in\mathbb{F}_{q^n}$ such that $$\begin{array}{c}
(a, b, c) = \lambda(x, y, z),\\
(r,s,t)=\sigma(x,y,z),\\
(u,v,w)=\mu(x,y,z),
\end{array}$$
and thus there are at most  $q^n$ triples $(u, v, w)$ for which the rank of $M$ equals $3$. The claim follows.
\end{proof}

Let us observe that the property of evasiveness cannot be significantly improved. Consider $\lambda_1,\lambda_2\in\mathbb{F}_q$ such that $$\begin{vmatrix}
    1&1&1\\
    {\lambda_1}&{\lambda_1}\qi&{\lambda_1}\qj\\
    {\lambda_2}&{\lambda_2}\qi&{\lambda_2}\qj
\end{vmatrix}\neq 0$$
and let
$$\begin{array}{ccl}
    h_1&:=&(1,0,0,1,1,0)\\
    h_2&:=&({\lambda_1},0,0,{\lambda_1}\qi,{\lambda_1}\qj,0)\\
    h_3&:=&({\lambda_2},0,0,{\lambda_2}\qi,{\lambda_2}\qj,0)\end{array}.$$
These three vectors are linearly independent. Now,  $h_4=(\mu,0,0,\mu\qi,\mu\qj,0)\in \langle h_1,h_2,h_3\rangle$  for each $\mu\in\mathbb{F}_{q^n}$, and thus $\us$ is   $(3,m)$-evasive with $m\geq n$.
 \begin{cor}\label{Cor:Finale}
       If $K_{\alpha ,\beta ,\gamma}^{I,J}$ is not an $A_{q,I,J}$-power in $\mathbb{F}_{q^n}$ and $J<({n-2})/{4}$ then $\us$ is indecomposable.
 \end{cor}
       \begin{proof} It follows from \cite[Lemma 3.4]{bartoli2022exceptional}.
    \end{proof}
 We conclude this section with our main result. 
 
\begin{theorem}
     Assume that $\gcd(n,I,J)=1$ and $K_{\alpha ,\beta ,\gamma}^{I,J}$ is not an $A_{q,I,J}$-power in $\mathbb{F}_{q^n}$. Then $\us$ is scattered and indecomposable in infinitely many extensions of $\mathbb{F}_q$.
\end{theorem}
     \begin{proof} From Corollary (\ref{cor1}) we have the existence of a sequence of positive integers $(m_k)_k$ such that $\us$ is scattered in 
     $\mathbb{F}_{q^{nm_k}}$ for every $k$. Moreover, there exists an 
     $m_{k_0}$ such that 
     $$J < \frac{n\cdot m_{k_0}-2}{4},$$ and, by Corollary \ref{Cor:Finale}, $\us$ is indecomposable in every extension $\mathbb{F}_{q^{nm_k}}$ with $m_k\ge m_{k_0}$.
     \end{proof}
     
\section{Equivalence issue}
This section is devoted to the determine a lower bound on the number of $\Gamma L_q(6,q^n)$-inequivalent scattered sets are contained in our family. Recall that, since $q=p^h$, the size of $Aut(\mathbb{F}_{q^n})$ is $hn$.
\begin{theorem}
    Let $I,J,I_0,J_0$ such that $\max(I+J,J+J_0,I_0+J_0)<n$, $I< J$ and $I_0< J_0$. The two sets $\us$ and $U_{\overline{\alpha} ,\overline{\beta} ,\overline{\gamma}}^{I_0,J_0,n}$ are not $\Gamma L(6,q^n)$-equivalent if $(I,J)\neq(I_0,J_0)$.
     \end{theorem}
    \begin{proof} The two sets $\us$ and $U_{\overline{\alpha} ,\overline{\beta} ,\overline{\gamma}}^{I_0,J_0,n}$ are $\Gamma L(6,q^n)$-equivalent if and only if there exist $\sigma\in Aut(\mathbb{F}_{q^n})$ and a matrix $M\in GL(6,\mathbb{F}_{q^n})$ such that
    $$\begin{pmatrix}
        a_{11}&a_{12}&a_{13}&a_{14}&a_{15}&a_{16}\\
        a_{21}&a_{22}&a_{23}&a_{24}&a_{25}&a_{26}\\
        a_{31}&a_{32}&a_{33}&a_{34}&a_{35}&a_{36}\\
        a_{41}&a_{42}&a_{43}&a_{44}&a_{45}&a_{46}\\
        a_{51}&a_{52}&a_{53}&a_{54}&a_{55}&a_{56}\\
        a_{61}&a_{62}&a_{63}&a_{64}&a_{65}&a_{66}
    \end{pmatrix}
    \begin{pmatrix}
        x^\sigma\\
        y^\sigma\\
        z^\sigma\\
        (\fax)^\sigma\\
        (\fbx)^\sigma\\
        (\fcx)^\sigma
    \end{pmatrix}
=\begin{pmatrix}
    u\\
    v\\
    w\\
    u^{q^{I_0}}+\overline{\alpha}v^{q^{J_0}}\\
    u^{q^{J_0}}+\overline{\beta}w^{q^{I_0}}\\
    v^{q^{I_0}}+\overline{\gamma}w^{q^{J_0}}\\
\end{pmatrix}
$$
Denote $\tilde{x}:=x\as,\tilde{y}:=y\as,\tilde{z}:=z\as,A:=\alpha\as,B:=\beta\as,\Gamma:=\gamma\as$. We obtain\begin{small}
    
$$
\begin{cases}
   a_{11}\tilde{x}+a_{12}\tilde{y}+a_{13}\tilde{z}+a_{14}(\tilde{x}\qi+A\tilde{y}\qj)+a_{15}(\tilde{x}\qj+B\tilde{z}\qi)+a_{16}(\tilde{y}\qi+\Gamma\tilde{z}\qj)=u\\
   a_{21}\tilde{x}+a_{22}\tilde{y}+a_{23}\tilde{z}+a_{24}(\tilde{x}\qi+A\tilde{y}\qj)+a_{25}(\tilde{x}\qj+B\as\tilde{z}\qi)+a_{26}(\tilde{y}\qi+\Gamma\tilde{z}\qj)=v\\
   a_{31}\tilde{x}+a_{32}\tilde{y}+a_{33}\tilde{z}+a_{34}(\tilde{x}\qi+\alpha\as\tilde{y}\qj)+a_{35}(\tilde{x}\qj+B\tilde{z}\qi)+a_{36}(\tilde{y}\qi+\Gamma\tilde{z}\qj)=w\\
   a_{41}\tilde{x}+a_{42}\tilde{y}+a_{43}\tilde{z}+a_{44}(\tilde{x}\qi+A\tilde{y}\qj)+a_{45}(\tilde{x}\qj+B\tilde{z}\qi)+a_{46}(\tilde{y}\qi+\Gamma\tilde{z}\qj)=u\qio+\overline{\alpha}v\qjo\\
   a_{51}\tilde{x}+a_{52}\tilde{y}+a_{53}\tilde{z}+a_{54}(\tilde{x}\qi+A\tilde{y}\qj)+a_{55}(\tilde{x}\qj+B\tilde{z}\qi)+a_{56}(\tilde{y}\qi+\Gamma\tilde{z}\qj)=u\qjo+\overline{\beta}w\qio\\
   a_{61}\tilde{x}+a_{62}\tilde{y}+a_{63}\tilde{z}+a_{64}(\tilde{x}\qi+A\tilde{y}\qj)+a_{65}(\tilde{x}\qj+B\tilde{z}\qi)+a_{66}(\tilde{y}\qi+\Gamma\tilde{z}\qj)=v\qio+\overline{\gamma}w\qjo\\
\end{cases}$$
\end{small}
Substituting, we obtain these three equations
\begin{align*}
a_{41}\tilde{x}+a_{42}\tilde{y}+a_{43}\tilde{z}+a_{44}(\tilde{x}\qi+A\tilde{y}\qj)+a_{45}(\tilde{x}\qj+B\tilde{z}\qi)+a_{46}(\tilde{y}\qi+\Gamma\tilde{z}\qj)&+\\-a_{11}\qio\tilde{x}\qio-a_{12}\qio\tilde{y}\qio-a_{13}\qio\tilde{z}\qio-a_{14}\qio(\tilde{x}^{q^{I+I_0}}+A\qio\tilde{y}^{q^{J+I_0}})&+\\-a_{15}\qio(\tilde{x}^{q^{J+I_0}}+B\qio\tilde{z}^{q^{I+I_0}})-a_{16}\qio(\tilde{y}^{q^{I+I_0}}+\Gamma\qio\tilde{z}^{q^{J+I_0}})&+\\-\overline{\alpha}(a_{21}\qjo\tilde{x}\qjo+a_{22}\qjo\tilde{y}\qjo+a_{23}\qjo\tilde{z}\qjo+a_{24}\qjo(\tilde{x}^{q^{I+J_0}}+A\qjo\tilde{y}^{q^{J+J_0}})&+\\+a_{25}\qjo(\tilde{x}^{q^{J+J_0}}+B\qjo\tilde{z}^{q^{I+J_0}})+a_{26}\qjo(\tilde{y}^{q^{I+J_0}}+\Gamma\qjo\tilde{z}^{q^{J+J_0}}))&=0\tag{a}\label{eq1}
 \end{align*}
 \begin{align*}
a_{51}\tilde{x}+a_{52}\tilde{y}+a_{53}\tilde{z}+a_{54}(\tilde{x}\qi+A\tilde{y}\qj)+a_{55}(\tilde{x}\qj+B\tilde{z}\qi)+a_{56}(\tilde{y}\qi+\Gamma\tilde{z}\qj)&+\\-a_{11}\qjo\tilde{x}\qjo-a_{12}\qjo\tilde{y}\qjo-a_{13}\qjo\tilde{z}\qjo-a_{14}\qjo(\tilde{x}^{q^{I+J_0}}+A\qjo\tilde{y}^{q^{J+J_0}})&+\\-a_{15}\qjo(\tilde{x}^{q^{J+J_0}}+B\qjo\tilde{z}^{q^{I+J_0}})-a_{16}\qjo(\tilde{y}^{q^{I+J_0}}+\Gamma\qjo\tilde{z}^{q^{J+J_0}})&+\\-\overline{\beta}(a_{31}\qio\tilde{x}\qio+a_{32}\qio\tilde{y}\qio+a_{33}\qio\tilde{z}\qio+a_{34}\qio(\tilde{x}^{q^{I+I_0}}+A\qio\tilde{y}^{q^{J+I_0}})&+\\+a_{35}\qio(\tilde{x}^{q^{J+I_0}}+B\qio\tilde{z}^{q^{I+I_0}})+a_{36}\qio(\tilde{y}^{q^{I+I_0}}+\Gamma\qio\tilde{z}^{q^{J+I_0}}))&=0\tag{b}\label{eq2}
 \end{align*}
\begin{align*}
a_{61}\tilde{x}+a_{62}\tilde{y}+a_{63}\tilde{z}+a_{64}(\tilde{x}\qi+A\tilde{y}\qj)+a_{65}(\tilde{x}\qj+B\tilde{z}\qi)+a_{66}(\tilde{y}\qi+\Gamma\tilde{z}\qj)&+\\-a_{21}\qio\tilde{x}\qio-a_{22}\qio\tilde{y}\qio-a_{23}\qio\tilde{z}\qio-a_{24}\qio(\tilde{x}^{q^{I+I_0}}+A\qio\tilde{y}^{q^{J+I_0}})&+\\-a_{25}\qio(\tilde{x}^{q^{J+I_0}}+B\qio\tilde{z}^{q^{I+I_0}})-a_{26}\qio(\tilde{y}^{q^{I+I_0}}+\Gamma\qio\tilde{z}^{q^{J+I_0}})&+\\-\overline{\gamma}(a_{31}\qjo\tilde{x}\qjo+a_{32}\qjo\tilde{y}\qjo+a_{33}\qjo\tilde{z}\qjo+a_{34}\qjo(\tilde{x}^{q^{I+J_0}}+A\qjo\tilde{y}^{q^{J+J_0}})&+\\+a_{35}\qjo(\tilde{x}^{q^{J+J_0}}+B\qjo\tilde{z}^{q^{I+J_0}})+a_{36}\qjo(\tilde{y}^{q^{I+J_0}}+\Gamma\qjo\tilde{z}^{q^{J+J_0}}))&=0\tag{c}\label{eq3}
 \end{align*}
 If there exists an element in $GL(6,\mathbb{F}_{q^n})$ such that (\ref{eq1}), (\ref{eq2}), and (\ref{eq3}) are satisfied for every $\tilde{x}, \tilde{y}, \tilde{z}$, it implies that these polynomials must be identically zero.
 \begin{itemize}
     \item $I\neq I_0,J_0$
     
     Considering the coefficients of $\tilde{x},\tilde{y},\tilde{z},\tilde{x}\qi,\tilde{y}\qi,\tilde{z}\qi$ in (\ref{eq1}) we obtain $$a_{41}=a_{42}=a_{43}=a_{44}=a_{45}=a_{46}=0.$$ Hence $M\not\in GL(6,\mathbb{F}_{q^n})$.
     \item $I=J_0$

     \begin{itemize}
         \item $J\neq I+I_0,2I$

         Considering the coefficients of $\tilde{x},\tilde{y},\tilde{z},\tilde{x}\qj,\tilde{y}\qj,\tilde{z}\qj$ in (\ref{eq1}) we obtain $a_{41}=a_{42}=a_{43}=a_{44}=a_{45}=a_{46}=0.$

         \item $J=I+I_0$ and $J+I_0=2I$

         Considering the coefficients of 
         $$\tilde{x},\tilde{y},\tilde{z},\tilde{x}\qj,\tilde{y}\qj,\tilde{z}\qj,\tilde{x}^{q^{I+J_0}},\tilde{y}^{q^{I+J_0}},\tilde{z}^{q^{I+J_0}},\tilde{x}^{q^{J+J_0}},\tilde{y}^{q^{J+J_0}},\tilde{z}^{q^{J+J_0}}$$ in (\ref{eq1}) we obtain $a_{41}=a_{42}=a_{43}=a_{44}=a_{45}=a_{46}=0.$

         \item $J=I+I_0$ and $J+I_0\neq 2I$

         Considering the coefficients of $\tilde{x},\tilde{y},\tilde{z},\tilde{x}\qj,\tilde{y}\qj,\tilde{z}\qj,\tilde{x}^{q^{J+I_0}},\tilde{y}^{q^{J+I_0}},\tilde{z}^{q^{J+I_0}}$ in (\ref{eq1}) we obtain $a_{41}=a_{42}=a_{43}=a_{44}=a_{45}=a_{46}=0.$

         \item $J=2I$ and $J\neq I+I_0$

         Considering the coefficients of $\tilde{x},\tilde{y},\tilde{z},\tilde{x}\qj,\tilde{y}\qj,\tilde{z}\qj,\tilde{x}^{q^{J+J_0}},\tilde{y}^{q^{J+J_0}},\tilde{z}^{q^{J+J_0}}$ in (\ref{eq1}) we obtain $a_{41}=a_{42}=a_{43}=a_{44}=a_{45}=a_{46}=0.$
         \end{itemize}
         Hence $M\not\in GL(6,\mathbb{F}_{q^n})$.
    \item $I=I_0$ and $J\neq J_0$
    \begin{itemize}
        \item $J\neq I+J_0$ and $J\neq 2I$

        Considering the coefficients of $\tilde{x},\tilde{y},\tilde{z},\tilde{x}\qj,\tilde{y}\qj,\tilde{z}\qj$ in (\ref{eq1}) we obtain $a_{41}=a_{42}=a_{43}=a_{44}=a_{45}=a_{46}=0.$

        \item $J= I+J_0$

        Considering the coefficients of $\tilde{x},\tilde{y},\tilde{z},\tilde{x}\qj,\tilde{y}\qj,\tilde{z}\qj,\tilde{x}^{q^{J+J_0}},\tilde{y}^{q^{J+J_0}},\tilde{z}^{q^{J+J_0}}$ in (\ref{eq1}) we obtain $a_{41}=a_{42}=a_{43}=a_{44}=a_{45}=a_{46}=0.$

        \item $J=2I$ and $J_0\neq 3I$

        Considering the coefficients of $\tilde{x},\tilde{y},\tilde{z},\tilde{x}\qj,\tilde{y}\qj,\tilde{z}\qj,\tilde{x}^{q^{I+J}},\tilde{y}^{q^{I+J}},\tilde{z}^{q^{I+J}}$ in (\ref{eq1}) we obtain $a_{41}=a_{42}=a_{43}=a_{44}=a_{45}=a_{46}=0.$

        \item $J=2I$ and $J_0= 3I$

        Considering the coefficients of $\tilde{x},\tilde{y},\tilde{z},\tilde{x}^{q^{2I}},\tilde{y}^{q^{2I}},\tilde{z}^{q^{2I}},\tilde{x}^{q^{5I}},\tilde{y}^{q^{5I}},\tilde{z}^{q^{5I}}$ in (\ref{eq1}) we obtain $a_{41}=a_{42}=a_{43}=a_{44}=a_{45}=a_{46}=0.$

    \end{itemize}
    Hence $M\not\in GL(6,\mathbb{F}_{q^n})$.
 \end{itemize}
 \end{proof}

 \begin{theorem}  
 Let $(I,J)$ such that $J<n/2$. Two sets $\us$ and $U_{\overline{\alpha} ,\overline{\beta} ,\overline{\gamma}}^{I,J,n}$ are $\Gamma L(6,q^n)$-equivalent if and only if $\exists$ $\sigma\in Aut(\mathbb{F}_{q^n})$ such that one among these three elements is a $q^{3K}-1$ power:$$\left(\frac{\overline{\alpha}}{\alpha^{\sigma}}\right)\left(\frac{\overline{\gamma}}{\gamma^{\sigma}}\right)^{q^k}\left(\frac{\beta^{\sigma}}{\overline{\beta}}\right)^{q^{2k}},\left(\frac{\overline{\gamma}}{\alpha^{\sigma}}\right)\left(\frac{1}{\gamma^{\sigma}\overline{\beta}}\right)^{q^k}\left({\overline{\alpha}}{\beta^{\sigma}}\right)^{q^{2k}},\left(\frac{1}{\alpha^{\sigma}\overline{\beta}}\right)\left(\frac{\overline{\alpha}}{\gamma^{\sigma}}\right)^{q^k}\left({\overline{\gamma}}{\beta^{\sigma}}\right)^{q^{2k}}.$$
 \end{theorem}
 \begin{proof}

 We examine again  polynomials (\ref{eq1}), (\ref{eq2}), and (\ref{eq3}) as in the proof of the previous theorem in the case where $I_0=I,J_0=J$ and investigate the conditions under which they vanish.
By checking the coefficients of $\tilde{x}, \tilde{y}, \tilde{z}, \tilde{x}^{q^{2I}}, \tilde{y}^{q^{2I}}, \tilde{z}^{q^{2I}}, \tilde{x}^{q^{2J}} \tilde{y}^{q^{2J}}, \tilde{z}^{q^{2J}}$  we obtain:$$
 \begin{array}{c}
a_{41}=a_{42}=a_{43}=a_{51}=a_{52}=a_{53}a_{61}=a_{62}=a_{63}=a_{14}=\\=a_{15}=a_{16}=a_{24}=a_{25}=a_{26}a_{34}=a_{35}=a_{36}=0.\end{array}$$
Therefore, the polynomials read:
$$
\begin{cases}
    \begin{array}{c}
     \tilde{x}\qi(a_{44}-a_{11}\qi)+\tilde{x}\qj(a_{45}-\overline{\alpha}a_{21}\qj)+\tilde{y}\qi(a_{46}-a_{12}\qi)+\\+\tilde{y}\qj(Aa_{44}-\overline{\alpha}a_{22}\qj)+
     \tilde{z}\qi(Ba_{45}-a_{13}\qi)+\tilde{z}\qj(\Gamma a_{46}-\overline{\alpha}a_{23}\qj)=0
         \end{array}\vspace{2mm}\\
    \begin{array}{c}
     \tilde{x}\qi(a_{54}-\overline{\beta}a_{31}\qi)+\tilde{x}\qj(a_{55}-a_{11}\qj)+\tilde{y}\qi(a_{56}-\overline{\beta}a_{32}\qi)+\\+\tilde{y}\qj(Aa_{54}-a_{12}\qj)+
     \tilde{z}\qi(Ba_{55}-\overline{\beta}a_{33}\qi)+\tilde{z}\qj(\Gamma a_{56}-a_{13}\qj)=0
         \end{array}\vspace{2mm}\\
    \begin{array}{c}
     \tilde{x}\qi(a_{64}-a_{21}\qi)+\tilde{x}\qj(a_{65}-\overline{\gamma}a_{31}\qj)+\tilde{y}\qi(a_{66}-a_{22}\qi)+\\+\tilde{y}\qj(Aa_{64}-\overline{\gamma}a_{32}\qj)+
     \tilde{z}\qi(Ba_{65}-a_{23}\qi)+\tilde{z}\qj(\Gamma a_{66}-\overline{\gamma}a_{33}\qj)=0.
         \end{array}
 \end{cases}$$
 These polynomials are identically zero if and only if$$
 \begin{array}{ccc}
    a_{44}=a_{11}\qi=\frac{\overline{\alpha}}{A}a_{22}\qj,  & a_{54}=\overline{\beta}a_{31}\qi=\frac{1}{A}a_{12}\qj, & a_{64}=a_{21}\qi=\frac{\overline{\gamma}}{A}a_{32}\qj, \\
    a_{45}=\frac{1}{B}a_{13}\qi=\overline{\alpha}a_{21}\qj,  & a_{55}=\frac{\overline{\beta}}{B}a_{33}\qi=a_{11}\qj, & a_{65}=\frac{1}{B}a_{23}\qi=\overline{\gamma}a_{31}\qj, \\
    a_{46}=a_{12}\qi=\frac{\overline{\alpha}}{\Gamma}a_{23}\qj,  & a_{56}=\overline{\beta}a_{32}\qi=\frac{1}{\Gamma}a_{13}\qj, & a_{66}=a_{22}\qi=\frac{\overline{\gamma}}{\Gamma}a_{33}\qj.
 \end{array}$$
     Thus $$
     \begin{array}{ccc}
        a_{11}=\left(\frac{\overline{\alpha}}{A}\right)^{q^{-I}}a_{22}\qk, & a_{22}=\left(\frac{\overline{\gamma}}{\Gamma}\right)^{q^{-I}}a_{33}\qk, & a_{33}=\left(\frac{B}{\overline{\beta}}\right)^{q^{-I}}a_{11}\qk,\\
        a_{21}=\left(\frac{\overline{\gamma}}{A}\right)^{q^{-I}}a_{32}\qk, & a_{32}=\left(\frac{1}{\Gamma \overline{\beta}}\right)^{q^{-I}}a_{13}\qk, & a_{13}=(B\overline{\alpha})^{q^{-I}}a_{21}\qk,\\
        a_{31}=\left(\frac{1}{A\overline{\beta}}\right)^{q^{-I}}a_{12}\qk, & a_{12}=\left(\frac{\overline{\alpha}}{\Gamma}\right)^{q^{-I}}a_{23}\qk, & a_{23}=(B\overline{\gamma})^{q^{-I}}a_{31}\qk.
        
     \end{array}$$
     and hence $$\begin{cases}
     a_{11}=\left(\frac{\overline{\alpha}}{A}\right)^{q^{-I}}\left(\frac{\overline{\gamma}}{\Gamma}\right)^{q^{k-I}}\left(\frac{B}{\overline{\beta}}\right)^{q^{2k-I}}a_{11}^{q^{3K}},\\
     a_{21}=\left(\frac{\overline{\gamma}}{A}\right)^{q^{-I}}\left(\frac{1}{\Gamma\overline{\beta}}\right)^{q^{k-I}}\left({\overline{\alpha}}{B}\right)^{q^{2k-I}}a_{21}^{q^{3K}},\\
     a_{31}=\left(\frac{1}{A\overline{\beta}}\right)^{q^{-I}}\left(\frac{\overline{\alpha}}{\Gamma}\right)^{q^{-I}}\left({\overline{\gamma}}{B}\right)^{q^{2k-I}}a_{31}^{q^{3K}}.
     \end{cases}$$
    By our assumptions, there exists $\sigma$ such that one of the three coefficients involved is a $q^{3k}-1$ power. Therefore, we can find one among $ a_{11}, a_{12}, a_{13}$ that is nonzero and satisfies the equation. We can  set the other two  to zero,  constructing a non-singular matrix whose coefficients make the polynomials we are studying identically zero. Thus, the composition of $\sigma$ and the found matrix gives us an equivalence between the two sets.

     Conversely, if for every $\sigma$ all three coefficients are not $q^{3K}-1$ powers, then the only solution to the system is $a_{11}, a_{12}, a_{13} = 0$ for every $\sigma$, making the matrix singular and the tow sets are not equivalent.
 \end{proof}

It is interesting to provide a lower bound on the number of inequivalent scattered sequences. To this aim, we make use of the previous theorem. To simplify the notation, we will denote 
\begin{eqnarray*}
K_{1,\sigma}&:=&\left(\frac{\overline{\alpha}}{\alpha^{\sigma}}\right)\left(\frac{\overline{\gamma}}{\gamma^{\sigma}}\right)^{q^k}\left(\frac{\beta^{\sigma}}{\overline{\beta}}\right)^{q^{2k}},\\
K_{2,\sigma}&:=&\left(\frac{\overline{\gamma}}{\alpha^{\sigma}}\right)\left(\frac{1}{\gamma^{\sigma}\overline{\beta}}\right)^{q^k}\left({\overline{\alpha}}{\beta^{\sigma}}\right)^{q^{2k}},\\
K_{3,\sigma}&:=&\left(\frac{1}{\alpha^{\sigma}\overline{\beta}}\right)\left(\frac{\overline{\alpha}}{\gamma^{\sigma}}\right)^{q^k}\left({\overline{\gamma}}{\beta^{\sigma}}\right)^{q^{2k}}.\\
\end{eqnarray*}

Let us fix a  triple $(\overline{\alpha},\overline{\beta},\overline{\gamma})$ and consider all the  $(\alpha, \beta, \gamma)$ such that the corresponding sets are equivalent. Let $C_1=\frac{\overline{\alpha}\overline{\gamma}^{q^K}}{\overline{\beta}^{q^{2K}}}$. Therefore, $K_{1,id} = C_1\frac{{\beta}^{q^{2K}}}{{\alpha}{\gamma}^{q^K}}$.
Given $\gamma$ and $\beta$, the function $\alpha \longmapsto C_1\frac{{\beta}^{q^{2K}}}{{\alpha}{\gamma}^{q^K}}$ is a permutation of $\mathbb{F}_q$. Since $\gamma$ and $\beta$ can vary in $(q^n-1)^2$ ways, $C_1\frac{{\beta}^{q^{2K}}}{{\alpha}{\gamma}^{q^K}}$ is a $(q^{3k}-1)$-power for $(q^n-1)^3/{\gcd(q^{3K}-1,q^n-1)}$ triples $(\alpha, \beta, \gamma)$.

An equivalence via $\sigma\neq id$ corresponds to an equivalence with $(\alpha^{\sigma}, \beta^{\sigma}, \gamma^{\sigma})$. Via the condition on $K_{1,\sigma}$, there are  at most $nh(q^n-1)^3/({\gcd(q^{3K}-1,q^n-1)})$ sets $U_{{\alpha} ,{\beta} ,{\gamma}}^{I,J,n}$ equivalent to $U_{\overline{\alpha} ,\overline{\beta} ,\overline{\gamma}}^{I,J,n}$.

Arguing analogously for  $K_{2,\sigma}$ and $K_{3,\sigma}$, we obtain that  $$\frac{3nh(q^n-1)^3}{\gcd(q^{3K}-1,q^n-1)}$$
is an upper bound for the number of sets $U_{{\alpha} ,{\beta} ,{\gamma}}^{I,J,n}$ equivalent to a fixed $U_{\overline{\alpha} ,\overline{\beta} ,\overline{\gamma}}^{I,J,n}$.

We now determine a lower bound on the number of distinct instances of $(\alpha, \beta, \gamma)$ giving rice to a scattered set $U_{{\alpha} ,{\beta} ,{\gamma}}^{I,J,n}$. The number of triples $(\alpha, \beta, \gamma) $ such that $\alpha\gamma^{q^K}\beta^{q^{2K}}$ is not a $q^{2K} + q^K + 1$ power can be determined as follows. Fixing $\beta$ and $\gamma$, the function $\alpha\longmapsto\alpha\gamma^{q^K}\beta^{q^{2K}}$ is a permutation. Therefore, there are $$q^n-1-\frac{q^n-1}{\gcd(q^{2K} + q^K + 1,q^n-1)}$$ choices for $\alpha\in \mathbb{F}_{q^n}^*$ such that $\alpha\gamma^{q^K}\beta^{q^{2K}}$ is not a $q^{2K} + q^K + 1$ power. Thus, there are in total 
$$(q^n-1)^3\left (1-\frac{1}{\gcd(q^{2K} + q^K + 1,q^n-1)}\right)$$ triples satisfying this condition. 

This shows that there are at least 
$$\frac{1}{3nh}\left(\gcd(q^{3K}-1,q^n-1)- \frac{\gcd(q^{3K}-1,q^n-1)}{\gcd(q^{2K} + q^K + 1,q^n-1)}\right)$$
equivalent classes.
Note that, when $3K\mid n$, we have that the above bound is roughly $\frac{q^{3K}-q^K}{3nh}$.

\begin{ex}
    Letting $q = 4, n = 12, J = 3$, and $I = 1$, we obtain $q^n-1 = 16777215, q^{2K}+q^K+1 = 273$, and $q^{3K}-1 = 4095$. There are at least $62$ inequivalent examples with these parameters.
\end{ex}

 \section*{ Acknowledgement } This work was  supported by the
Research Project of MIUR (Italian Office for University and
Research) ``Strutture Geometriche, Combinatoria e loro Applicazioni''.


\section*{Conflict of interest}
 On behalf of all authors, the corresponding author states that there is no conflict of interest. 
\bibliographystyle{abbrv}
\bibliography{biblio.bib}

\end{document}